\documentclass[11pt, reqno]{amsart}

\usepackage{amsmath, amssymb, amsthm}
\usepackage{mathtools}
\usepackage{esint}
\usepackage[colorlinks=true, urlcolor=blue, linkcolor=blue, citecolor=blue, pdfstartview=]{hyperref}
\usepackage{color}

\newtheorem*{theorem*}{Theorem}
\newtheorem{theorem}{Theorem}
\newtheorem{lemma}{Lemma}
\newtheorem{proposition}{Proposition}
\newtheorem{definition}{Definition}
\newtheorem{corollary}{Corollary}

\newtheorem{remark}{Remark}

\numberwithin{equation}{section}

\begin{document}

\title[Isoperimetric inequalities]{Some sharp isoperimetric-type inequalities on Riemannian manifolds}

\author{Kwok-Kun Kwong}
\address{Department of Mathematics, National Cheng Kung University, Tainan City 701, Taiwan}
\email{kwong@mail.ncku.edu.tw}

\address{Current address: School of Mathematics and Applied Statistics, University of Wollongong, NSW 2522, Australia}
\email{kwongk@uow.edu.au}


\begin{abstract}
We prove some sharp isoperimetric type inequalities for domains with smooth boundary on Riemannian manifolds. For example, using generalized convexity, we show that among all domains with a lower bound $l$ for the cut distance and Ricci curvature lower bound $(n-1)k$, the geodesic ball of radius $l$ in the space form of curvature $k$ has the largest area-to-volume ratio. A similar but reversed inequality holds if we replace a lower bound on the cut distance by a lower bound of the mean curvature. As an application we show that $C^2$ isoperimetric domains in standard space forms are balls. Generalized convexity also provides a simple proof of Toponogov theorem. We also prove another isoperimetric inequality involving the extrinsic radius of a domain when the curvature of the ambient space is bounded above. We then extend this inequality in two directions: one involves the higher order mean curvatures, and the other involves the Hausdorff measure of the cut locus.
\end{abstract}
\maketitle
\section{Introduction}
Let $\Omega$ be a Riemannian manifold with smooth boundary. We define the area-to-volume ratio to be $\frac{|\partial \Omega|}{|\Omega|}$, where $|\partial \Omega|$ is the area of $\partial \Omega$ and $|\Omega|$ is the volume of $\Omega$. In this note, we prove a sharp isoperimetric type inequality under a lower bound of the Ricci curvature and the cut distance. More precisely, we prove that among all domains with a lower bound $l$ for the cut distance and Ricci curvature lower bound $(n-1)k$, the geodesic ball of radius $l$ in the space form of curvature $k$ has the largest area-to-volume ratio (Corollary \ref{cor: iso1}):

\begin{theorem}[Corollary \ref{cor: iso1}]\label{introthm1}
Suppose $(\Omega^n, g)$ is a complete Riemannian manifold with smooth compact boundary. Assume the Ricci curvature of $\Omega$ satisfies $\mathrm{Ric}\ge (n-1)k g$ and the cut distance of $\Omega$ satisfies $c(\Omega)= l$. Then
\begin{align*}
\frac{ \mathrm{Area}(\partial \Omega)}{\mathrm{Vol}(\Omega)}\le \frac{\mathrm{Area}(\mathbb S^{n-1}_k(l))} {\mathrm{Vol}\left(\mathbb B^n_k (l)\right)}.
\end{align*}
The equality holds if and only if $\Omega$ is isometric to $\mathbb B^n_k(l)$.
\end{theorem}
Here, $\mathbb B^n_{k}(l)$ and $\mathbb S^{n-1}_{k}(l)$ denotes the geodesic ball and the geodesic sphere of radius $l$ in the $n$-dimensional simply-connected space form $M_k$ of curvature $k$, respectively. In this paper, $n=\mathrm{dim}(\Omega)$ is assumed to be at least two.

The above result can be compared to the Levy-Gromov isoperimetric inequality \cite{G}, which states that
\begin{theorem}[Levy-Gromov isoperimetric inequality]
Suppose $(M^n, g)$ is a closed manifold which has Ricci curvature $\mathrm{Ric}\ge (n-1)k$, $k>0$. Assume $\Omega$ is a domain in $M$ which has a smooth boundary $\partial \Omega$ and let $B$ be a geodesic ball in $M_k=\mathbb S^n (1 /\sqrt{k})$, the $n$-sphere of radius $\frac{1}{\sqrt{k}}$ in $\mathbb R^{n+1}$ (which has curvature $k$), such that
$\frac{\mathrm{Vol}(\Omega)}{\mathrm{Vol}(M)}
=\frac{\mathrm{Vol}(B)}{\mathrm{Vol} (M_k)}$.
Then
\begin{align}\label{ineq: LG}
\frac{\mathrm{Area}(\partial \Omega)}{\mathrm{Vol}(M)}
\ge\frac{\mathrm{Area}(\partial B)}{\mathrm{Vol} (M_k)}.
\end{align}
\end{theorem}
There are two features of our result. One is that unlike the Levy-Gromov isoperimetric inequality, we allow the Ricci curvature lower bound to be arbitrary (not necessarily positive) and still get a sharp inequality. The second one, which seems more interesting, is that we provide a lower bound for the volume instead of an upper bound (of course the bound cannot depend only on the boundary area, but also on its cut distance). This is in contrast with the classical isoperimetric inequality, the Levy-Gromov isoperimetric inequality and the Bishop-Gromov volume comparison, all of which provide (roughly speaking) an upper bound of the volume of a domain either in terms of its boundary area, or in terms of the volume of its counterpart in the comparison model space.
More precisely, note that \eqref{ineq: LG} can be written as $\frac{\mathrm{Area}(\partial \Omega)}{\mathrm{Vol}(\Omega)} \ge\frac{\mathrm{Area}(\partial B)}{\mathrm{Vol} (B)}$. Comparing this with the Theorem \ref{introthm1}, the inequality sign is ``reversed''.
The proof of Theorem \ref{introthm1} makes use of a generalized notion of convexity, called $\mathcal F_k$-convexity. Interestingly, $\mathcal F_k$-convexity leads to a simple proof of Toponogov theorem (Theorem \ref{topo}). The Laplacian comparison theorem (Theorem \ref{thm: vol c} \eqref{mc comp} when $\Omega=B(p, r)$), Bonnet-Myers' theorem (Theorem \ref{BM}), Bishop-Gromov volume comparison theorem (Corollary \ref{cor: BG}) and a comparison result of Li \cite{li2014sharp} and Ge \cite{ge2015comparison} (Corollary \ref{li ge}) also follow as corollaries of our main result.

We also prove a sharp Heintze-Karcher-Ros type inequality which is closely related to Theorem \ref{introthm1}:
\begin{theorem}[Theorem \ref{v h}]\label{introthm1.5}
Let $\Omega^n$ be a compact Riemannian manifold with $C^2$ boundary and $\mathrm{Ric}\ge (n-1)kg$ on $\Omega$. If $k\le 0$, assume the normalized mean curvature $H_1>\sqrt{-k}$ on $\partial \Omega$. Then
\begin{align}\label{ros'}
\mathrm{Vol}(\Omega)
\le&\int_{\partial \Omega} h_k(l_k(H_1))dS.
\end{align}
The equality holds if and only if $\Omega$ is isometric to $\mathbb B^n_k(l)$ for some $l$.
\end{theorem}
The functions $l_k$ and $h_k$ are explicit. Indeed $l_k(\lambda) $ is the radius of the sphere in $M_k$ which has principal curvature $\lambda$ and $h_k$ is the function which makes \eqref{ros'} an equality when $\Omega$ is a geodesic ball in $M_k$. In particular, this implies that among all compact domains $\Omega$ with $\mathrm{Ric}\ge (n-1)k$ and with $H_1\ge \lambda_k(l)$ on the boundary, $\mathbb B^n_k(l)$ is the unique maximizer of $\frac{|\Omega|}{|\partial \Omega|}$ (Corollary \ref{maximizer}). Using Theorem \ref{introthm1.5}, we give a proof that the $C^2$ isoperimetric domains in $M_k$ are geodesic balls (Theorem \ref{isop domain}).

Next, we prove an isoperimetric-isodiametric inequality. Let $M$ be a complete Riemannian manifold (without boundary). For $\Omega\subset M$, the extrinsic radius is $\mathrm{rad}_M(\Omega) :=\inf\{r>0: \Omega\subset B(x_0, r) \textrm{ for some } x_0\in M\}$, where $B(x_0, r)$ is the geodesic ball centered at $x_0$ with radius $r$ in $M$. We can prove the following result.
\begin{theorem}[Theorem \ref{thm: iso2}]\label{introthm2}
Let $M$ be a complete Riemannian manifold (without boundary). Suppose $(\Omega^n, g)$ is a domain in $M$ with compact piecewise $C^1$ boundary and with $\mathrm{rad}_M(\Omega)<\infty$. Suppose the sectional curvature of $M$ is bounded above by $k$. Let $\mathrm{rad}_M(\Omega)=L$, which is assumed to be $\le \frac{\pi}{\sqrt{k}}$ if $k>0$. Then
\begin{align*}
\frac{ \mathrm{Area}(\partial \Omega)}{\mathrm{Vol}(\Omega)}\ge \frac{\mathrm{Area}(\mathbb S^{n-1}_k(L))} {\mathrm{Vol}\left(\mathbb B^n_k (L)\right)}.
\end{align*}
The equality holds if and only if $\Omega$ is isometric to $\mathbb B^n_k(L)$.
\end{theorem}

As an application, we give a lower bound of the Cheeger's constant (Corollary \ref{cheeger}). Combining Theorem \ref{introthm1} and Theorem \ref{introthm2}, we also prove a version of Santal{\'o}-Ya{\~n}ez theorem \cite{average} for family of domains in the hyperbolic space and the Euclidean space, with assumptions on the cut distance instead of convexity of the boundary (Corollary \ref{SY}).

Not all domains are contained in a geodesic ball.
If we replace the geodesic balls by metric balls in the above definition, we obtain an isoperimetric inequality on the area ($n-1$ dimensional Hausdorff measure) of the cut locus instead.
\begin{theorem}[Corollary \ref{cor cut}]\label{introthm3}
Suppose $M$ is an $n$-dimensional closed (compact without boundary) manifold and its curvature is bounded from above by $k$. Let $x_0\in M$. If $k>0$, we further assume that $\mathrm{rad}(x_0) <\frac{\pi}{\sqrt{k}}$. Then
\begin{align*}
\frac{\mathrm{Vol}(M)}{2 \mathcal H^{n-1}(\mathrm{Cut}(x_0)) }\le \frac{ \mathrm{Vol}(\mathbb B^n_k(L)) }{ \mathrm{Area}(\mathbb S^{n-1}_k(L)) }
\end{align*}
where $L=\mathrm{rad}(x_0)$.
\end{theorem}
The rest of this paper is organized as follows. In Section \ref{sec1}, we prove Theorem \ref{introthm1}. In order to prove Theorem \ref{introthm1}, we introduced the concept of generalized convexity in Subsection \ref{sec convex}. As an application, we give a simple proof of Toponogov theorem in Subsection \ref{toponogov}. Theorem \ref{introthm1} is then proved in Subsection \ref{main}. In some cases, we can replace the assumptions on the cut distance with assumptions on the curvature of the boundary and such results are presented in Subsection \ref{alt}. In Section \ref{hk ineq}, we prove Theorem \ref{introthm1.5} and we apply it to show that $C^2$ isoperimetric domains in simply connected space forms are geodesic balls. In Subsections \ref{extrin}, \ref{higher}, Theorem \ref{introthm2} together with their generalizations to higher order mean curvatures are proved. In Subsection \ref{cut}, Theorem \ref{introthm2} is generalized to the case where the domain is not necessarily contained in a geodesic ball, and in particular Theorem \ref{introthm3} is proved.

{\sc Acknowledgments}:
Part of this research is inspired by a question of Pengzi Miao. We would like to thank him for encouragement.
We would also like to thank Kostiantyn Drach, Hojoo Lee, Liren Lin, Mu-Tao Wang and Ye-Kai Wang for discussions. The research of the author is partially supported by Ministry of Science and Technology in Taiwan under grant MOST 106-2115-M-006-017-MY2.

\section{Isoperimetric inequalities involving the cut distance}\label{sec1}
\subsection{Generalized convex functions}\label{sec convex}
In order to prove our main results, we need some knowledge of the theory of generalized convex functions. We first fix the notation. Let $k\in \mathbb R$ be fixed throughout this note. Let $\lambda\in \mathbb R$ such that
\begin{align}\label{eq: lambda}
\lambda>
\begin{cases}
0 &\textrm{ if }k=0\\
-\infty &\textrm{ if }k>0\\
\sqrt{-k}&\textrm{ if }k<0
\end{cases}
\end{align}
and let
\begin{equation}\label{eq: l}
l=l_k(\lambda)=
\begin{cases}
\frac{1}{\lambda}\quad &\textrm{if }k=0\\
\frac{1}{\sqrt k} \cot^{-1}\left(\frac{\lambda}{\sqrt k}\right)\quad &\textrm{if }k>0\\
\frac{1}{\sqrt {-k}} \coth^{-1}\left(\frac{\lambda}{\sqrt {-k}}\right)\quad &\textrm{if }k<0.
\end{cases}
\end{equation}
This is equivalent to
\begin{equation}\label{eq: lambda2}
\lambda=\lambda_k(l)=
\begin{cases}
\frac{1}{l}\quad &\textrm{if }k=0\\
\sqrt{k}\cot\left(\sqrt{k}l\right) \quad &\textrm{if }k>0\\
\sqrt {-k} \coth \left({\sqrt {-k}}l\right) \quad &\textrm{if }k<0.
\end{cases}
\end{equation}
We also define $\sigma_{k, \lambda}(t)$ to be the solution of
\begin{align*}
\sigma''(t)+k \sigma(t)=0, \sigma(0)=1, \sigma'(0)=-\lambda.
\end{align*}
More explicitly,
\begin{equation}\label{f bar}
\sigma_{k, \lambda}(t)
=\begin{cases}
1-\lambda t \quad &\textrm{if }k=0\\
\cos \left(\sqrt k t\right)-\frac{\lambda }{\sqrt k} \sin \left(\sqrt k t \right) \quad &\textrm{if }k>0\\
\cosh \left(\sqrt {-k}t\right) -\frac{\lambda }{\sqrt {-k}}\sinh \left(\sqrt {-k} t\right) \quad &\textrm{if }k<0.
\end{cases}
\end{equation}

The geometric meaning of $l$ and $\lambda$ are as follows. In the $n$-dimensional simply-connected space form $M_k$ with curvature $k$, the geodesic sphere with radius $l$, $\mathbb S^{n-1}_k(l)$, is umbilical with principal curvatures equal to $\lambda$.

\begin{proposition}\label{prop: 1}
\begin{enumerate}
\item
Suppose $\lambda$ and $l$ be defined by \eqref{eq: lambda} and
\eqref{eq: l} respectively.
Assume $f, \overline f:[0, l]\to \mathbb R$ satisfy
\begin{align*}
\begin{cases}
&f''\le -kf\textrm{ on }[0, l], \quad
f(0)=1, \quad f(l)\ge 0\\
&{\overline f}'' =-k \overline f\textrm{ on }[0, l], \quad
\overline f(0)=1, \quad \overline f(l)= 0.
\end{cases}
\end{align*}
Then $f(t)\ge \overline f(t)$ on $[0, l]$.
\item
If $k\le 0$ and $f:[0, \infty)\to \mathbb R$ satisfies
\begin{align*}
f''\le -k f, \quad f(0)=1, \quad f(t)>0
\end{align*}
Then $f(t)\ge \overline f(t)$, where
$
\overline f(t)
=
\begin{cases}
1\quad &\textrm{if }k=0\\
e^{-\sqrt{-k}t}\quad &\textrm{if }k<0.
\end{cases}
$

\end{enumerate}
\end{proposition}

To prove Proposition \ref{prop: 1}, we introduce the concept of $\mathcal F$-convexity. Let $\mathcal F_k$ be the space of all solutions to the second order differential equations
\begin{align}\label{eq: DE}
h''=-k h\quad \textrm{ on }[0, l].
\end{align}
More explicity,
$\mathcal F_k =\mathrm{span} \{c_k, s_k\}$
where
\begin{equation}\label{eq: sk}
c_k(t)=
\begin{cases}
1&\textrm{ if }k=0\\
\cos \left(\sqrt {k}t\right) &\textrm{ if }k>0\\
\cosh \left(\sqrt {-k}t\right) &\textrm{ if }k<0
\end{cases}
\quad \textrm{ and }\quad
s_k(t)=
\begin{cases}
t&\textrm{ if }k=0\\
\frac{1}{\sqrt{k}}\sin \left(\sqrt{k}t\right)&\textrm{ if }k>0\\
\frac{1}{\sqrt{-k}}\sinh \left(\sqrt{-k}t\right)&\textrm{ if }k<0.
\end{cases}
\end{equation}

For fixed $k$, a function $f$ defined on $[0, l]$ is said to be $\mathcal F_k$-convex if
\begin{enumerate}
\item
for any $0\le x_1<x_2 \le l$,
there is a unique solution $h$ to \eqref{eq: DE} with $h(x_i)=f(x_i)$, and
\item
$f\le h$ on $[x_1, x_2]$.
\end{enumerate}

A function $f$ is said to be $\mathcal F_k$-concave if $-f$ is $\mathcal F_k$-convex.
By \cite[Theorem 2]{peixoto1949generalized} (cf. also \cite[Lemma 1.1]{clement1953generalized}), we have the following characterization of $\mathcal F_k$-concavity.
\begin{lemma}\label{lem: concave}
Suppose $f$ is $C^2$ on $[0, l]$ and $
\begin{vmatrix}
s_k(x_1) &c_k(x_1)\\
s_k(x_2) &c_k(x_2)
\end{vmatrix}\ne 0
$ for any distinct $x_1, x_2\in [0, l]$. Then $f''\le -k f$ if and only if $f$ is $\mathcal F_k$-concave.
\end{lemma}

For later use, we generalize this lemma to $\mathcal F_k$-convex/concave functions in the support sense. In fact, we show that $\mathcal F_k$-convexity in the support sense can be reduced to convexity in the classical sense by suitably transforming the function. We expect that this reduction has other applications, since much more is known about convex functions than $\mathcal F_k$-convex functions.
We say a function $f$ is $\mathcal F_k$-convex in the support sense if $f$ is continuous and $f''+k f\ge 0$ in the support sense. i.e. for any $p$ and $\varepsilon>0$, there exists a neighborhood $U$ of $p$ and a $C^2$ function $f_\varepsilon$ on $U$ (called a support function) such that $f_\varepsilon(p)=f(p)$, $f_\varepsilon\ge f$ on $U$ and $f_\varepsilon''(p)+k f_\varepsilon(p)> -\varepsilon$. (cf. \cite[p. 279]{petersen2006riemannian} and \cite{calabi1958extension} for more general definition).

\begin{lemma}\label{convex support}
Let $\mathrm{ct}_k(t)=\frac{c_k(t)}{s_k(t)}$ and $I$ be an open subinterval in $(0, l)$, where $l\le\frac{\pi}{\sqrt{k}}$ if $k>0$. Assume $\phi$ is $\mathcal F_k$-convex on $I$ in the support sense, then
\begin{enumerate}
\item\label{convex 1}
$\sqrt{k+y^2}\, \phi (\mathrm{ct}_k^{-1}(y))$ is a convex function on $ \mathrm{ct}_k(I)$ in the classical sense.
\item\label{convex 2}
$\phi$ is $\mathcal F_k$-convex.
\end{enumerate}
\end{lemma}
\begin{proof}
\eqref{convex 1} is proved for $\mathcal F_k$-convex functions in the classical sense in \cite[Lemma 1.1]{clement1953generalized}. For convenience we sketch it here. For $(x_1, x_2)\subset I$, let $h=\alpha s_k+\beta c_k\in \mathcal F_k$ such that $h(x_i)=\phi(x_i)$. Then the condition $h(x)\ge \phi(x)$ for $x\in(x_1, x_2)$ is equivalent to $\alpha+\beta \mathrm{ct_k}(x)\ge \frac{\phi(x)}{s_k(x)}, $ which is equivalent to $\alpha+\beta y\ge \frac{\phi}{s_k}\circ \mathrm{ct}_k^{-1}(y)=
\sqrt{k+y^2}\, \phi (\mathrm{ct}_k^{-1}(y)) =:\psi(y)$ for $y\in \mathrm{ct}_k((x_1, x_2))$.

Now, suppose $\phi$ is $\mathcal F_k$-convex in the support sense, and let $\phi_\varepsilon$ be the corresponding support function from below on a neighborhood of $x_0$. Direct computation shows that the second derivative of
$\sqrt{k+y^2}\, \phi_\varepsilon (\mathrm{ct}_k^{-1}(y))$ is
${(k+y^2)^{-\frac{3}{2}}} \left[\phi_\varepsilon''(\mathrm{ct}_k^{-1}(y))+k \phi_\varepsilon(\mathrm{ct}_k^{-1}(y))\right] $, which is greater than $- (k+\mathrm{ct}_k(x_0)^2)^{-\frac{3}{2}} \, \varepsilon $ at $\mathrm{ct}_k(x_0)$.
This shows that $\psi$ is convex in the support sense, as the two other conditions are easy to check.

To see that $\psi$ is convex, fix $\varepsilon>0$, let $h_\varepsilon(y)=\varepsilon y^2$ and $\psi_{\varepsilon, y_0}$ be a support function on a neighborhood $U$ of $y_0$ with $\psi_{\varepsilon, y_0}''({y_0})>-\varepsilon$. Then $\psi+h_\varepsilon\ge \psi_{\varepsilon, y_0}+h_{\varepsilon}$ and $(\psi_{\varepsilon, y_0}+h_\varepsilon)''({y_0})>\varepsilon>0$. Hence we can assume $\psi_{\varepsilon, y_0}+h_\varepsilon$ is a convex function on $U$ by shrinking $U$ if necessary. As $y_0$ is arbitrary, by Lemma \ref{convex support2}, $\psi+h_\varepsilon$ is a convex function. By taking $\varepsilon\to 0$, $\psi$ is also convex.

\eqref{convex 2} then follows by reversing the steps in paragraph 1.
\end{proof}

\begin{lemma}\label{convex support2}
Let $f$ be a continuous function defined on an open interval $I$ such that for each $y_0$, there exists a convex function $g$ defined on a neighborhood $U$ of $y_0$ with $g(y_0)=f(y_0)$ and $g(y)\le f(y)$ on $U$. Then $f$ is convex.
\end{lemma}

\begin{proof}
Let $(y_1, y_2)\subset I$, and by subtracting a linear function whose values agree with $f$ at $y_i$, we can assume $f(y_i)=0$. It remains to show that $f\le 0$ on $(y_1, y_2)$. If there is an interior maximum point $p\in(y_1, y_2)$, then $f$ is supported below by a convex function $g$ near $p$, which also attains a maximum at $p$. But convexity of $g$ implies that $g$ is locally constant near $p$, and thus $f$ is constant near $p$. This shows that on $(y_1,y_2)$, $f$ is either constant or there is no local maximum, and hence $f\le 0$.
\end{proof}

\subsection{A proof of the Toponogov theorem using $\mathcal F_k$-convexity}\label{toponogov}
In this subsection, we illustrate how the proof of the Toponogov theorem can be simplified by using Lemma \ref{convex support}. The proof does not require consideration of different cases by the sign of the curvature or the types of triangles. This subsection is independent of later sections.

The following version of Toponogov's theorem is taken from \cite{meyer1989toponogov}. Please refer to \cite{meyer1989toponogov} for other versions that can be deduced from \eqref{ii} and for further applications. Our aim is just to give a more unified proof. We denote the length of a geodesic segment by $|\cdot|$.
\begin{theorem}\label{topo}
Let $(M, g)$ be a complete Riemannian manifold with sectional curvature bounded below by $k$.
\begin{enumerate}
\item\label{i}
Let $p_{0}, p_{1}, q$ be three distinct points in $M$, $\gamma$ be a geodesic from $p_{0}$ to $p_{1}$ and $\gamma_{i} $ be minimal geodesics from $q$ to $p_{i}$, $i=0, 1$. Suppose $|\gamma| \leq\left|\gamma_0\right|+\left|\gamma_1\right|$ and $|\gamma| \leq \frac{\pi}{\sqrt{k}}$ in the case $k>0$.
Then there exists a corresponding comparison triangle $\overline{p}_{0}, \overline{p}_{1}, \overline{q} $ in the two-dimensional $M_{k}^{2}$ with corresponding geodesics $\overline{\gamma}_{0}, \overline{\gamma}_{1}, \overline{\gamma}$ which are all minimal of lengths $\left|\overline{\gamma}_{i}\right|=\left|\gamma_{i}\right|, |\overline{\gamma}|=|\gamma|$.
\item\label{ii}
We have the inequality $d(q, \gamma(t))\ge d(\overline q, \overline \gamma(t))$ for $t\in [0, |\gamma|]$.
\end{enumerate}
\end{theorem}

\begin{proof}
Part \eqref{i} is proved in \cite[Theorem 2.2]{meyer1989toponogov} so we prove \eqref{ii}. Let $l=|\gamma|$. If $k>0$, as shown in \cite[Theorem 2.2 Step 2]{meyer1989toponogov}, by a limiting argument we can assume $l<\frac{\pi}{\sqrt{k}}$. Let $\gamma(t):[0, l]\to M$ be a unit-speed geodesic from $p_0$ to $p_1$ and $\overline \gamma(t):[0, l]\to M_k^2$ be a minimal geodesic from $\overline p_0$ to $\overline p_1$. Let $r=d(q, \cdot)$, and $\overline r= d(\overline q, \cdot)$ be the corresponding function in $ M_k^2$.

Define $\rho_k(\tau):=\int_{0}^{\tau}s_k(t)dt$. In the following, all the differential inequalities are in the support sense. By Hessian comparison theorem (\cite[p. 342]{petersen2006riemannian}), $\nabla ^2 (\rho_k\circ r)\le c_k(r) g$. This implies $f(t):=(\rho_k\circ r)(\gamma(t))$ satisfies $f''(t)\le c_k(r(\gamma(t)))$ and $\overline f(t):=(\rho_k\circ \overline r)(\overline \gamma(t))$ satisfies $\overline f''(t)= c_k(\overline r(\overline \gamma(t)))$. Integrating the identity $s_k''+ks_k=0$ gives $c_k(r)+k \rho_k(r)=1$, which implies $f''(t)+k f(t)\le 1$ and $\overline f''(t)+k \overline f(t)= 1$.
This shows that function $\phi(t):=f(t)-\overline f(t)$ satisfies $\phi''-k \phi\le 0$. We also have the boundary conditions $\phi(0)=\phi(l)$. By Lemma \ref{convex support}, $\phi$ is $\mathcal F_k$-concave. Since $0$ is the unique function in $\mathcal F_k$ with the same boundary values, we have $\phi\ge 0$.
We conclude that $r(\gamma(t))\ge \overline r (\overline \gamma(t))$.
\end{proof}

\subsection{An isoperimetric inequality involving the cut distance}\label{main}
Suppose $(\Omega^n, g)$ is a Riemannian manifold with smooth boundary. We will assume that $\Omega$ is orientable.
Let $p\in \Sigma=\partial \Omega$ and $N(p)$ be the inward unit normal of $\Sigma$. We define the cut function $c(p):=\sup\{t: d(\exp_p(tN(p)), \Sigma)=t\} $ and $\displaystyle c(\Omega):=\inf _{p\in \partial \Omega}c(p)$.
We say a Riemannian manifold with boundary is complete if it is complete as a metric space.

Sometimes to simplify notations, we let $|\cdot|$ to denote the $k$-dimensional volume for any $k$, so e.g. $|\Omega|$ to denote $\mathrm{Vol}(\Omega)$ and $|\partial \Omega|$ denotes $\mathrm{Area}(\partial \Omega)$, which is the $(n-1)$-dimensional volume of $\partial \Omega$.

The following function is essential in stating our results.
Let
\begin{equation}\label{eq: h}
h_k(t)=\frac{|\mathbb B^n_k(t)|}{|\mathbb S^{n-1}_k(t)|}.
\end{equation}
Explictly,
\begin{align}\label{eq: h explicit}
h_k(t) =
\begin{cases}
\displaystyle t/n \quad &\textrm{if }k=0\\
\displaystyle \frac{\displaystyle \int_{0}^{t}\sin^{n-1}\left(\sqrt{k}r\right)dr}{\sin^{n-1}(\sqrt{k}t)}\quad &\textrm{if }k>0\\
\displaystyle \frac{\displaystyle \int_{0}^{t}\sinh^{n-1}\left(\sqrt{-k}r\right)dr}{\sinh^{n-1}(\sqrt{-k}t)}\quad &\textrm{if }k<0.
\end{cases}
\end{align}

We first begin with some simple properties of $h_k$.
\begin{lemma}\label{lim h}
For $k\le 0$, $\lim _{t\to\infty}\frac{1}{h_k(t)}=(n-1)\sqrt{-k}$.
\end{lemma}
\begin{lemma}\label{lem: mono}
The function
$h_k(t)=\frac{|\mathbb B^n_k(t)|}{|\mathbb S^{n-1}_k(t)|}$
is increasing on $(0, l)$, where $l =
\begin{cases}
\infty\quad &\textrm{if }k\le 0\\
\frac{\pi}{\sqrt{k}}\quad &\textrm{if }k>0.
\end{cases}
$

\end{lemma}

\begin{proof}

From \eqref{eq: h explicit}, $1/h_k(t)$ is of the form
\begin{align*}
\frac{\int_{0}^{t}({s_k}^{n-1})'(r)dr}{\int_{0}^{t}s_k^{n-1}(r)dr}
\end{align*}
where $s_k(r)$ is given by \eqref{eq: sk}. By Lemma \ref{lem: dec} below, since it is readily checked that $\frac{({s_k}^{n-1})'(r)}{s_k^{n-1}(r)}= (n-1) \frac{s_k'(r)}{s_k(r)}$ is decreasing, so is $1/h_k$.
\end{proof}

\begin{lemma}[\cite{daicomparison} Lemma 1.4.10]\label{lem: dec}
If $f, g$ are two continuous functions such that $f(t)/g(t)$ is decreasing for $t>0$ and $g$ is positive, then
$\displaystyle \frac{\int_{0}^{t}f(s)ds}{\int_{0}^{t}g(s)ds}$
is decreasing for $t>0$.
\end{lemma}

\begin{lemma}\label{lem: convex}
Let $g(t)$ be a smooth function such that $g(0)=0$, $g(t)>0$ and $1/g(t)$ is strictly convex for $t>0$.
Let
$h(t):=\frac{\int_{0}^{t}g(r)dr}{g(t)}$
and $u(t)=\log (g(t))$. Assume also that $\displaystyle \lim_{t \downarrow 0} \frac{u'(t)g(t)}{u'(t)^2-u''(t)}=0$.
\begin{enumerate}
\item
If $2(u'')^2-u'u'''> 0$, then $h(t)$ is a strictly convex function.
\item
If $2(u'')^2-u'u'''<0$, then $h(t)$ is a strictly concave function.
\end{enumerate}
\end{lemma}

\begin{proof}
Since
$h(t)=e^{-u(t)} \int_{0}^{t}e^{u(r)}dr $,
we compute
$h'(t)= -u'(t)e^{-u(t)}\int_{0}^{t}e^{u(r)}dr+1$
and
\begin{align*}
h''(t)
=&-u''(t)e^{-u(t)}\int_{0}^{t}e^{u(r)}dr+u'(t)^2 e^{-u(t)}\int_{0}^{t}e^{u(r)}dr-u'(t)\\
=&\left(u'(t)^2-u''(t)\right)e^{-u(t)}\int_{0}^{t}e^{u(r)}dr -u'(t).
\end{align*}
Note that
$\left(1/g\right)''=(e^{-u})''=e^{-u}\left(u'^2-u''\right)>0$,
so the condition $h''(t)> 0$ is equivalent to
\begin{align*}
\int_{0}^{t}e^{u(r)}dr> \frac{u'(t)e^{u(t)}}{u'(t)^2-u''(t)}.
\end{align*}
At $t=0$, both sides of the above is $0$ (by continuous extension).
Thus by taking the derivative again, the above can be implied by
\begin{align*}
0< e^{u(t)}- \left(\frac{u'(t)e^{u(t)}}{u'(t)^2-u''(t)}\right)'
= e^{u} \left[\frac{2(u'')^2-u'u'''}{\left(u'^2-u''\right)^2} \right],
\end{align*}
which is true by our assumption in case 1. The concave case is of course similar.
\end{proof}

\begin{lemma}\label{lem: h convex}
\begin{enumerate}
\item
If $k>0$, then $h_k(t)$ is strictly convex on $\left(0, \frac{\pi}{\sqrt{k}}\right)$.
\item
$h_0$ is linear.
\item
If $k<0$, then $h_k(t)$ is strictly concave.
\end{enumerate}

\end{lemma}

\begin{proof}
The function $h_0$ is obviously linear.
It is easy to see that $\frac{1}{s_k(t)^{n-1}}$ is convex for all $k$ and that if $k>0$, then $u(t)=m\log (s_k(t))$ satisfies $2u''^2-u'u'''=2 k^2 m^2 \csc ^2(\sqrt{k} t)>0$ on $(0, \frac{\pi}{\sqrt{k}})$ and so $h_k(t)$ is convex by Lemma \ref{lem: convex}. Similarly if $k<0$, then $2u''^2-u'u'''=-2 k^2 m^2 \mathrm{csch} ^2(\sqrt{-k} t)<0$ and so $h_k$ is concave.
\end{proof}

We can now state our first main result.
\begin{theorem}\label{thm: vol c}
Suppose $(\Omega^n, g)$ is a complete Riemannian manifold with smooth boundary. Assume the Ricci curvature of $\Omega$ satisfies $\mathrm{Ric}\ge (n-1)k g$. Then
\begin{enumerate}
\item\label{mc comp}
The mean curvature of $\partial \Omega$ satisfies $H(p)\le (n-1) \lambda_k(c(p))$ (see \eqref{eq: lambda2}) and
\begin{equation}\label{vol h}
\mathrm{Vol}(\Omega) \ge\int_{\partial \Omega}h_k(c(p)) dS.
\end{equation}
\item
Suppose the equality in \eqref{vol h} holds and $\mathrm{Vol}(\Omega)<\infty$. Then outside the cut locus of $\partial \Omega$, $\Omega$ is isometric to $\{(p, t): p\in \Sigma, 0\le t<c(p)\}$ with metric
$g(p, t)=dt^2+ ({\overline f}_{c(p)}(t))^2 g|_{\Sigma}(p)$, where
$\overline f_{c(p)}(t)=\sigma_{k, \lambda_p}(t)$
with $\lambda_p=\lambda_k(c(p))$.
\item
Suppose $\Omega$ is compact. Then the equality holds if and only if $\Omega$ is isometric to $\mathbb B^n_k(l)$ for some $l$.
\end{enumerate}
\end{theorem}
Before proving Theorem \ref{thm: vol c}, we give some simple corollaries.
Recall the Chebyshev's inequality: if $h$ is a non-negative non-decreasing function on a measure space $(X, \mu)$,
then
$\mu(\{x:f(x)\ge t\})\le \frac{1}{h(t)}\int_X h\circ f d\mu$.
Combining Lemma \ref{lem: mono}, Theorem \ref{thm: vol c}, Lemma \ref{lim h} and the Chebyshev’s inequality, we have the following corollaries.
\begin{corollary}
Suppose $\mathrm{Ric}\ge (n-1)kg$ on $\Omega$, then $h_k(t)\mathcal H^{n-1}(\{p\in \partial \Omega: c(p)\ge t\}) \le \mathrm{Vol}(\Omega)$. In particular, $\{p\in\partial \Omega:c(p)\le t_0\}\ne \emptyset$, where $h_k(t_0)=\frac{\mathrm{Vol}(\Omega)}{\mathrm{Area}(\partial \Omega)}$.
\end{corollary}

\begin{corollary}
Suppose $\mathrm{Ric}\ge 0$ on $\Omega$ and
$ \mathrm{Vol}(\Omega) <\infty$, then $\mathcal H^{n-1}(\{p\in \partial \Omega: c(p)=\infty\})=0$.
\end{corollary}
\begin{corollary}
Suppose $\mathrm{Ric}\ge (n-1)kg$ on $\Omega$ for some $k<0$, then $\mathcal H^{n-1}(\{p\in \partial \Omega: c(p)=\infty\})\le \sqrt{-k}(n-1)\mathrm{Vol}(\Omega)$.
\end{corollary}

\begin{lemma}\label{c}
Let $k>0$. With the assumption in Theorem \ref{thm: vol c}, we have $c(p)<\frac{\pi}{\sqrt{k}}$ for all $p\in \partial \Omega$.
\end{lemma}
\begin{proof}
We may assume $k=1$. As in Bonnet-Myers' theorem, it is easy to see that $c(p)\le \pi$. Suppose $c(p)=\pi$. Let $\Sigma_t$ be the parallel hypersurfaces $\{\exp_{q}(t N(q)): q\in U\}$ where $U$ is a small neighborhood of $p$ in $\partial \Omega$ and $N$ is the inward unit normal. By the first variation formula, the mean curvatures $H_{\Sigma_t}(\exp_{p}(t N(p)))\to H(p)$ as $t\to 0^+$. On the other hand, by \cite[Eqn. 1.8]{ge2015comparison}, $H_{\Sigma_t}(\exp_{p}(t N(p)))\le (n-1)\cot (\pi-t)\to -\infty$, a contradiction.
\end{proof}
We now prove Theorem \ref{thm: vol c}.
\begin{proof}[Proof of Theorem \ref{thm: vol c}]
\begin{enumerate}
\item
Let $m=n-1$, $p\in \partial \Omega$ and $N(p)$ be the inward unit normal of $\partial \Omega$.
Define $F(p, t)$ by $dV(\exp_p(tN(p)))=F(p, t)dt\wedge dS$, where $dS$ be the induced area form on $\Sigma$. The second variation formula then reads (\cite[Equation 1.5]{li1993lecture})
\begin{equation*}
\label{eq: 2nd}
\frac{\partial ^2}{\partial t^2}F(p, t)= -\mathrm{Ric}(\partial_t, \partial_t)F(p, t)+ \left(H(p, t) ^2 -|A(p, t)|^2\right) F(p, t),
\end{equation*}
where $H(p, t)$ and $A(p, t)$ denotes the mean curvature and the second fundamental form (w.r.t. outward normal) of the hypersurface $\Sigma_t= \psi_t(\Sigma)$ at the point $\psi_t(p)$, and $\psi_t(p)=\exp_{p}(tN(p))$. We use the convention that $A=-\nabla N$ on $\Sigma$ and $H=\mathrm{tr}(A)$.

Let us write $F_p(t):=F(p, t)$, regarded as a family of functions in $t$.
Using the Cauchy-Schwarz inequality and the first variation formula
\begin{equation}\label{eq: 1var}
{F_p}'=\frac{\partial F}{\partial t}=- H(p, t) F_p,
\end{equation}
we have
\begin{equation}\label{ineq: 2var}
\begin{split}
{F_p}''
\le& \frac{m-1}{m}H(p, t)^2 F_p-\mathrm{Ric}(\partial_t, \partial_t)F_p\\
=& \frac{m-1}{m}\frac{{{F_p}'}^2}{F_p} -\mathrm{Ric}(\partial_t, \partial_t)F_p\\
\le& \frac{m-1}{m}\frac{{{F_p}'}^2}{F_p} - mk F_p.
\end{split}
\end{equation}

For convenience, we let $f_p(t)=F_p(t)^{\frac{1}{m}}$.
Then from \eqref{ineq: 2var},
we have
$${f_p}''\le -k f_p. $$

Taking also the boundary conditions $ F_p(0)=1 $ and $ F_p(c(p))\ge 0$ into considerations (as there cannot be any focal point along $t\mapsto \exp_p(tN(p))$ before $c(p)$), we have
$$
\begin{cases}
{f_p}''\le-k f_p\\
f_p(0)=1\\
f_p(c(p))\ge 0.
\end{cases}
$$

For any $p$, let $\overline f_{c(p)}(t)$ be the solution of the ODE (if $k>0$, $c(p)< \frac{\pi}{\sqrt{k}}$ by Lemma \ref{c})
\begin{equation*}
\begin{cases}
{\overline f_{c(p)}}''=-k \overline f_{c(p)}. \\
\overline f_{c(p)}(0)=1, \\
\overline f_{c(p)}(c(p))=0.
\end{cases}
\end{equation*}

More explicitly, if we define $\lambda_p=\lambda_k(c(p))$, then $\overline f_{c(p)}$ is given by
$\overline f_{c(p)}(t)=\sigma_{k, \lambda_p}(t)$ (see \eqref{f bar}).
(If $c(p)=\infty$, we take the limit $\displaystyle \lambda_p=\lim_{l\to \infty}\lambda(l)$ and $\overline f_{c(p)}$ is still well-defined.)
By Proposition \ref{prop: 1}, we have $f_p(t)\ge \overline f_{c(p)}(t)$ on $[0, c(p)]$. So \eqref{eq: 1var} gives $H(p)\le m\lambda_k(c(p))$.

To complete the proof, note that
\begin{equation}\label{ineq: vol}
\begin{split}
\mathrm{Vol}(\Omega)
=& \int_{\mathbb \partial \Omega}\int_{0}^{c(p)}f_p(t)^mdt dS(p)\\
\ge& \int_{\mathbb \partial \Omega}\int_{0}^{c(p)}\overline f_{c(p)}(t)^m dt d S(p)\\
=&\int_{\partial \Omega}h_k(c(p)) dS.
\end{split}
\end{equation}
The last line follows because the volume of $\mathbb B^n_k(r)$ is exactly given by
\begin{align}\label{hk int}
|\mathbb B^n_k(r)|
=\int_{\mathbb S^m_k(r)}\int_{0}^{r}\overline f_r(t)^m dt \, dS
=|\mathbb S^m_k(r)|\int_{0}^{r}\overline f_r(t)^m dt
\end{align}
and so $\displaystyle h_k(r)=\int_{0}^{r}\overline f_r(t)^m dt$.
\item
Suppose the equality holds and $\mathrm{Vol}(\Omega)<\infty$.
Then $F(p, t)=\left(\overline f_{c(p)}(t)\right)^m$ and $\Sigma_t$ is umbilical by \eqref{ineq: 2var}. The umbilicity implies that $\nabla ^2 t|_{\Sigma_t}=\frac{\overline {f}_{c(p)}'(t)}{\overline {f}_{c(p)}(t)} g|_{\Sigma_t(p)}$.

Let $h:=g|_{\Sigma_t}$, regarded as a family of metrics on $\Sigma$ and $x^i$ be a local coordinates on $\Sigma$. Then
\begin{align*}
\frac{d}{dt}h_{ij}=h(\nabla _i\partial_t, \partial_j)+h(\partial_i, \nabla _j \partial_t)=2\nabla ^2 t(\partial_i, \partial_j)=2\frac{\overline {f}_{c(p)}'(t)}{\overline {f}_{c(p)}(t)}h_{ij}.
\end{align*}
From this it is easy to see that $h=({\overline f}_{c(p)}(t))^2 g|_{\Sigma}$ and so $g(p, t)=dt^2+ ({\overline f}_{c(p)}(t))^2 g|_{\Sigma}(p)$.
\item
If $\Omega$ is compact and the equality holds, then as before $\partial \Omega$ is umbilical and \eqref{eq: 1var} gives $c(p)= l_k(H_1(p))$ for all $p\in \partial \Omega$, where $H_1=\frac{H}{m}$. The equality then becomes $\int_{\partial \Omega} h_k(l_k(H_1))dS= \mathrm{Vol}(\Omega)$, which is the equality case of Theorem \ref{v h}. Therefore $\Omega$ is isometric to $\mathbb B^n_k(l)$ for some $l$.

\end{enumerate}
\end{proof}
Interestingly, the above argument also leads to a proof of the Bonnet-Myers' theorem.
\begin{theorem}\label{BM}
If $M^n$ is a compact Riemannian manifold with $\mathrm{Ric}\ge (n-1)k g$ ($k>0$), then $\mathrm{diam}(M)\le \frac{\pi}{\sqrt{k}}$.
\end{theorem}

\begin{proof}
Without loss of generality assume $k=1$. With the notation in the proof of Theorem \ref{thm: vol c}, the argument there actually shows that for any $l<\pi$, if $c(p)\ge l$, then $f_p(t)\ge \overline f_l(t)$ on $[0, l]$. By \eqref{eq: 1var} and \eqref{f bar}, this implies $H(p)\le (n-1) \cot l$.

Now, if there exists $p_0\in M$ such that $\max d (p_0, \cdot) >\pi$. Take a small geodesic ball $B(p_0, \varepsilon)$ such that there exists $p\in \partial \Omega$ with $c(p)\ge \pi$, where $\Omega:=M\setminus B(p_0, \varepsilon)$. From the above, the mean curvature of $\partial B(p_0, \varepsilon)$ at $p$ is infinite, a contradiction.
\end{proof}
\begin{corollary} \label{cor: iso1}
Suppose $(\Omega^n, g)$ is a complete Riemannian manifold with smooth compact boundary. Assume the Ricci curvature of $\Omega$ satisfies $\mathrm{Ric}\ge (n-1)k g$.
\begin{enumerate}
\item \label{case: k>0}
If $k\ge 0$ and the average of the cut function $\displaystyle \fint_{\partial \Omega}c(p)dS\ge l$, then
\begin{align*}
\frac{ \mathrm{Area}(\partial \Omega)}{\mathrm{Vol}(\Omega)}\le \frac{\mathrm{Area}(\mathbb S^{n-1}_k(l))} {\mathrm{Vol}\left(\mathbb B^n_k (l)\right)}.
\end{align*}
\item \label{case: k<0}
If $k<0$ and $c(\Omega)\ge l$,
then
\begin{align*}
\frac{ \mathrm{Area}(\partial \Omega)}{\mathrm{Vol}(\Omega)}\le \frac{\mathrm{Area}(\mathbb S^{n-1}_k(l))} {\mathrm{Vol}\left(\mathbb B^n_k (l)\right)}.
\end{align*}
\end{enumerate}
In both cases, the equality holds if and only if $\Omega$ is isometric to $\mathbb B^n_k(l)$.
\end{corollary}
\begin{proof} [Proof of Corollary \ref{cor: iso1}]
The inequality in case \ref{case: k<0} follows from Theorem \ref{thm: vol c} and Lemma \ref{lem: mono}. Then function $h_0(t)$ is linear and by Lemma \ref{lem: h convex}, $h_k$ is strictly convex when $k>0$. So when $k\ge 0$, by Theorem \ref{thm: vol c} and Jensen inequality,
\begin{align*}
\mathrm{Vol}(\Omega)\ge
\int_{\partial \Omega} h_k(c(p))dS
\ge |\partial \Omega|h_k\left({\fint}_{\partial \Omega}c(p)dS\right)
\end{align*}
where $\fint_{\partial \Omega}=\frac{1}{|\partial \Omega|}\int_{\partial \Omega}$.
In other words,
\begin{align*}
\frac{ \mathrm{Area}(\partial \Omega)}{\mathrm{Vol}(\Omega)}\le \frac{\mathrm{Area}(\mathbb S^{n-1}_k(l))} {\mathrm{Vol}\left(\mathbb B^n_k (l)\right)}.
\end{align*}

Suppose the equality holds.
Then in both cases, we deduce from \eqref{ineq: vol} that $c(p)=l$ for all $p\in \partial \Omega$ and $f=\overline f_l$. Therefore $F(t)=\left(\overline f_l(t)\right)^{n-1}$ and so from \eqref{eq: 1var}, the mean curvature of $\partial \Omega$ is exactly the mean curvature of $\mathbb S^{n-1}_k(l)$ in $\mathbb B^n_k(l)$, which is equal to $(n-1)\lambda$. Moreover, clearly there exists $q\in \Omega$ such that $d(q, \partial \Omega)=l$. It then follows from \cite[Theorem A (2)]{kasue1983ricci} that $\Omega$ is isometric to $\mathbb B^n_k(l)$ (note that $l$ is exactly $C_1(k, -\lambda)$ in \cite{kasue1983ricci}).
\end{proof}
We give an alternative proof of the following comparison and rigidity result of Li \cite{li2014sharp} and Ge \cite{ge2015comparison}.
\begin{corollary}\label{li ge}
Suppose $(\Omega^n, g)$ is a complete Riemannian manifold with smooth boundary. Suppose the Ricci curvature of $\Omega$ satisfies $\mathrm{Ric}\ge (n-1)k g$ and the mean curvature of $\partial \Omega$ satisfies $H\ge (n-1)\lambda$, where we assume $\lambda>\sqrt{-k}$ when $k\le 0$. Then $\sup_{x\in \Omega}d(x, \partial \Omega)\le l_k(\lambda)$. If $\partial \Omega$ is compact, then the equality holds if and only if $\Omega$ is isometric to $\mathbb B^n_k(l_k(\lambda))$.
\end{corollary}

\begin{proof}
Suppose $d(x, \partial \Omega)>l_k(\lambda)$ for some $x\in \Omega$, then $c(p)>l_k(\lambda)$ for some $p\in \partial \Omega$.
By Theorem \ref{thm: vol c} \eqref{mc comp}, $H(p)\le (n-1)\lambda_k(c(p))<(n-1)\lambda_k(l_k(\lambda))=(n-1)\lambda$, which is a contradiction.

Assume $\partial \Omega$ is compact and the equality holds. Then $\Omega$ is compact as $\sup _{x\in \Omega}d(x, \partial \Omega)<\infty$. Thus there exists $x_0\in \Omega$ with $d(x_0, \partial \Omega)=l_k(\lambda)$. By \cite[Theorem A (2)]{kasue1983ricci}, $\Omega$ is isometric to $\mathbb B^n_k(l_k(\lambda))$.
\end{proof}

Let us single out a weaker version of the $k=0$ case of
Corollary \ref{cor: iso1} here, as $h_0$ can be easily written down.
\begin{corollary}\label{cor1}
Suppose $(\Omega^n, g)$ has non-negative Ricci curvature, then
\begin{align*}
c(\Omega)|\partial \Omega| \le n|\Omega|.
\end{align*}
The equality holds if and only if $\Omega$ is a Euclidean ball.
\end{corollary}
\begin{remark}
\begin{enumerate}
\item
By inspecting the proof of Theorem \ref{thm: vol c}, we can actually generalize the estimate to
$\mathrm{Vol}(\Omega_\rho)\ge \int_{\partial \Omega} j_k(c(p), \rho)dS(p)$,
where $j_k(r, \rho)=\frac{1}{s_k(r)^m} \int_{\max\{r-\rho, 0\}}^{r}s_k(t)^m dt $ and $\Omega_\rho=\{x\in \Omega: d(x, \partial \Omega)\le \rho\}$.
\item
A special case of Corollary \ref{cor1}, where $\Omega$ is a metric ball, is proved in \cite{mondino2017isoperimetric} (Theorem 1.1).
It is easy to see that Corollary \ref{cor1} is not true if $\Omega$ does not have non-negative Ricci curvature. For example, let $\Omega$ be the revolution surface obtained by rotating the graph of $e^{x-L}$ on $[0, L]$ about the $x$-axis. Then independent of $L$, $|\partial \Omega|\ge 2\pi$, $c(\Omega)\to \infty$ as $L\to \infty$ but the area of $\Omega $ is bounded from above by a constant independent of $L$.
Clearly, $\Omega$ is negatively curved and does not satisfy the assumptions of Corollary \ref{cor1}.
\end{enumerate}
\end{remark}

As an illustration of Corollary \ref{cor: iso1},
if $\Omega_r$ is a family of smooth domains such that either $c(\Omega_r)=r$ when $k<0$, or $\fint _{\partial \Omega_r} c(p)=r$ when $k\ge 0$. Then
\begin{equation}\label{omega B}
\frac{|\Omega_r|}{| \partial \Omega_r|} \ge \frac{|\mathbb B^n_k(r)|}{| \mathbb S^{n-1}_k(r)|}.
\end{equation}
In particular, if $k_1\ge k_2=k$, $\Omega_r=\mathbb B^n_{k_1}(r)$ and $r<\mathrm{diam}(M_{k_1})$, then
\begin{equation*}
h_{k_1}(r)\ge h_{k_2}(r).
\end{equation*}

Examples of $\Omega_r$ satisfying \eqref{omega B} include geodesics balls in $M$, or by rescaling a smooth domain in the Euclidean space.
This also gives an alternative proof of the Bishop-Gromov volume comparison theorem.
\begin{corollary}\label{cor: BG}
Suppose $M^n$ is a complete Riemannian manifold with $\mathrm{Ric}\ge (n-1)kg$. Let $B(p, r)$ be a geodesic ball of radius $r$ in $M$, then
$\frac{|B(p, r)|}{|\mathbb B^n_k(r)|}$ is monotonic decreasing.
\end{corollary}
\begin{proof}
Observe that when $\Omega_r=B(p, r)$, \eqref{omega B} can be rewritten as
$\frac{d }{d r}\log \left(\frac{|B(p, r)|}{|\mathbb B^n_k(r)|}\right)\le 0$.
\end{proof}

Another corollary is a volume comparison theorem (or, in the dual sense, area comparison theorem if an upper bound of the volume is imposed instead of a lower bound of the boundary area).
\begin{corollary}\label{cor: vol}
Suppose $(\Omega^n, g)$ is a complete Riemannian manifold with smooth boundary. Assume the Ricci curvature of $\Omega$ satisfies $\mathrm{Ric}\ge (n-1)k g$. Suppose
there exists $l>0$ such that
$\mathrm{Area}(\partial \Omega)\ge\mathrm{Area}(\mathbb S^{n-1}_k(l))$ and
$c(\Omega)\ge l$.
Then
\begin{align*}
\mathrm{Vol}(\Omega)\ge
\mathrm{Vol}\left(\mathbb B^n_k (l)\right).
\end{align*}
The equality holds if and only if $\Omega$ is isometric to $\mathbb B^n_k(l)$.
\end{corollary}

\begin{remark}
The assumption on the cut distance in Corollary \ref{cor: vol} is necessary, even if we impose more stringent conditions on the boundary.

First of all, it is easy to see that the condition on the area alone is insufficient, as we can always find an ellipsoid in $\mathbb R^n$ with arbitrarily large boundary area but arbitrarily small volume.

Now we consider another case. Suppose $(\Omega^n, g)$ has boundary $\Sigma$ such that it has Ricci curvature $\mathrm{Ric}\ge 0$ and $\Sigma$ is isometric to $\mathbb S^{n-1}$ and such that its boundary has mean curvature $H\le n-1$, it is still not true that
$\mathrm{Vol}(\Omega)\ge \mathrm{Vol}(\mathbb B^n). $

For a long and thin ellipse $\Sigma$ lying on $\mathbb R^2\subset \mathbb R^3$ with circumference $2\pi$, we can add a convex cap $\Omega$ on it in the upper half space such that $\Sigma$ is totally geodesic in $\Omega$, i.e. $k_g=0$, but it has area less than $\pi$ (because the ``thin ellipse'' can have arbitrary small area and the cap can be ``arbitrarily short'').
\end{remark}
Corollary \ref{cor: iso1} can also be combined with Bishop's area or volume comparison, or Levy-Gromov isoperimetric inequality to produce some other results, some of them are recorded in the first version of the arXiv preprint.
We note that by modifying \eqref{ineq: vol}, Corollary \ref{cor: iso1} can be easily generalized to a weighted version, which may have independent interest.
\begin{proposition}
[Weighted isoperimetric inequality]
With the same assumptions as in Corollary \ref{cor: iso1}, suppose $\phi(\rho)$ is a positive continuous function and let $\rho_{\Sigma}$ and $\rho_{\mathbb S}$ denote the distance to $\Sigma=\partial \Omega$ and to $\mathbb S^{n-1}_k(l)$ respectively. Then
\begin{align*}
\frac{\mathrm{Area}(\partial \Omega)}{\displaystyle \int_{\Omega} \phi(\rho_{\Sigma })dV}\le
\frac{\mathrm{Area}(\mathbb S^{n-1}_k(l))}{\displaystyle \int_{\mathbb B^n_k(l)} \phi(\rho_{\mathbb S})dV}.
\end{align*}
The equality holds if and only if $\Omega$ is isometric to $\mathbb B^n_k(l)$.
\end{proposition}

\subsection{Results with alternative assumptions}\label{alt}
In practice, $c(\Omega)$ is harder to estimate than the mean curvature or the second fundamental form. It turns out that if $\Omega$ is a domain in a space form, then it is possible to replace the assumption on the cut distance with assumptions on the curvature of the boundary.

In this subsection, we assume $\partial \Omega$ is compact.
For $p\in \partial \Omega$, define
$$\mathrm{Focal}(p):=\sup\{t>0: d(\exp_p(tN(p)), \cdot): \partial \Omega\to \mathbb R \textrm{ attains a local minimum at $p$}\}. $$
Let $\displaystyle \mathrm{Focal}(\Omega):=\min_{p\in\partial \Omega}\mathrm{Focal}(p)$.
Clearly, $c(\Omega)\le \mathrm{Focal}(\Omega)$. It is often easier to compute or estimate $\mathrm{Focal}(\Omega)$ than to estimate $c(\Omega)$.
Under some additional conditions, we can have $c(\Omega)=\mathrm{Focal}(\Omega)$.
Some of these conditions are discussed in \cite{H}. (In \cite{H}, $c(\Omega)$ is called the rolling radius of $\Omega$.)
We directly record these results here for the convenience of the reader.

\begin{theorem}[\cite{H} Theorem 4.3]\label{thm: H3}
If $\mathrm{Ric} \ge (n-1)k>0$ on $\Omega$ and $\mathrm{NorWid}(\Omega)\ge \frac{\pi}{\sqrt{k}}$, then
$$c(\Omega) = \mathrm{Focal}(\Omega). $$
Here, $\mathrm{NorWid}(\Omega)$ is the normal width of $\Omega$, defined by
\begin{align*}
\mathrm{NorWid}(\Omega):=\min_{p\in \partial \Omega} \sup\{t: \exp_p(s N(p)) \in \Omega\textrm{ for }0< s< t\}.
\end{align*}
\end{theorem}

\begin{theorem}[\cite{H} Theorem 4.4]\label{thm: H4}
If either
\begin{enumerate}
\item
$\mathrm{Ric} \ge (n-1)kg$ on $\Omega$ and $H\ge (n-1)\sqrt{-k}$ for some $k<0$, or
\item
$\mathrm{Ric} \ge 0$ on $\Omega$ and $H > 0$.
\end{enumerate}
Then
$c(\Omega) = \mathrm{Focal}(\Omega)$.
\end{theorem}

\begin{theorem}[\cite{H} Theorem 4.5]\label{thm: H5}
Assume that $\Omega$ is compact and the sectional curvature of $\Omega$
satisfies $\mathrm{Sec}\ge k$ for some $k\le 0$.
Let $\lambda_i$ be the principal curvatures of $\partial \Omega$.
If either
\begin{enumerate}
\item
$k<0$ and the number of the $\lambda_i$ satisfying
$\lambda_i\ge \sqrt{-k}$ is greater than or equal to $\frac{n}{2}$,
or
\item
$k=0$ and the number of the $\lambda_i$ satisfying
$\lambda_i\ge h_0$ is greater than or equal to $\frac{n}{2}$ for some $h_0>0$.
\end{enumerate}
Then
$ c(\Omega) = \mathrm{Focal}(\Omega)$.
\end{theorem}

With the above results, we can replace the assumption in Corollary \ref{cor: iso1} on the lower bound of the cut distance with the lower bound of the focal length under some additional assumptions.

\begin{proposition}\label{prop: focal}
Suppose the sectional curvature of $\Omega$ is bounded above by $k$ and the second fundamental form of $\partial \Omega$ is bounded above by $\lambda$, where $\lambda$ satisfies \eqref{eq: lambda}. Then
$\mathrm{Focal}(\Omega)\ge l_k(\lambda)$,
where $l_k$ is given by \eqref{eq: l}.
\end{proposition}
The reader may read the first version of the arXiv preprint for details of the proof.
We now replace the assumption in Corollary \ref{cor: iso1} with other assumptions.
Combining Corollary \ref{cor: iso1}, Proposition \ref{prop: focal}, Theorems \ref{thm: H3}, \ref{thm: H4} and \ref{thm: H5}, we have the following results.

\begin{corollary}\label{cor12}
Suppose $\Omega$ has constant curvature $k$ and either
\begin{enumerate}
\item
$k>0$ and $\mathrm{NorWid}(\Omega)\ge \frac{\pi}{\sqrt{k}}$, or
\item
$k<0$ and $H\ge (n-1)\sqrt{-k}$, or
\item
$k= 0$ and $H > 0$.
\end{enumerate}
Assume that the second fundamental form of $\partial \Omega$ is bounded above by $\Lambda$, where $\Lambda>\sqrt{-k}$ if $k\le 0$ and let $l=l_k(\Lambda)$.
Then
$h_k(l)|\partial \Omega|\le |\Omega|$.
The equality holds if and only if $\Omega$ is $\mathbb B^n_k(l)$.
\end{corollary}

\begin{corollary}
Suppose $\Omega$ has constant curvature $k$.
If either
\begin{enumerate}
\item
$k<0$ and
$\left|\{i: \lambda_i(p)\ge \sqrt{-k}\}\right|\ge \frac{n}{2}$ for all $p\in \partial \Omega$, or
\item
$k=0$ and there exists $\lambda>0$ such that $ \left|\{i: \lambda_i(p)\ge \lambda\}\right|\ge \frac{n}{2}$ for all $p\in \partial \Omega$.
\end{enumerate}
Assume that the second fundamental form of $\partial \Omega$ is bounded above by $\Lambda$ as in Corollary \ref{cor12}. Then
$h_k(l)|\partial \Omega|\le |\Omega|$
where $l=l_k(\Lambda)$. The equality holds if and only if $\Omega$ is $\mathbb B^n_k(l)$.
\end{corollary}
\section{A Heintze-Karcher-Ros type inequality and isoperimetric domains in $M_k$}\label{hk ineq}

In this section, we prove a Heintze-Karcher-Ros type inequality which is closely related to Theorem \ref{thm: vol c}.
The following inequality is implicit in \cite{HK}. We define the normalized mean curvature $H_1=\frac{H}{n-1}$.
\begin{theorem}\label{v h}
Let $\Omega^n$ be a compact Riemannian manifold with $C^2$ boundary and $\mathrm{Ric}\ge (n-1)kg$ on $\Omega$. If $k\le 0$, assume $H_1>\sqrt{-k}$ on $\partial \Omega$. Then
\begin{align}\label{ros}
\mathrm{Vol}(\Omega)
\le&\int_{\partial \Omega} h_k(l_k(H_1))dS.
\end{align}
The equality holds if and only if $\Omega$ is isometric to $\mathbb B^n_k(l)$ for some $l$.
\end{theorem}

\begin{proof}
We use the notation in the proof of Theorem \ref{thm: vol c}.
By \cite[Cor. 3.3.2]{HK}, we have $F(p, t)\le (c_k(t)-H_1(p)s_k(t))^{n-1}$. On the other hand, by the second variation formula, as $c(p)$ is bounded by the focal length, we have
$c(p)\le l_k(\max k_i(p))\le l_k(H_1(p))$ (note that $l_k$ is decreasing, cf. also \cite[Theorem 0.1]{ge2015comparison}).
So
\begin{equation*}
\begin{split}
\mathrm{Vol}(\Omega)=&\int_{\partial \Omega}\int_{0}^{c(p)} F(p, t)dt dS(p)\\
\le&\int_{\partial \Omega}\int_{0}^{l_k(H_1(p))} \left(c_k(t)-H_1(p)s_k(t)\right)^{m} dt dS(p).
\end{split}
\end{equation*}
As in \eqref{hk int},
\begin{equation*}
\int_{0}^{l_k(H(p))} \left(c_k(t)-H_1(p)s_k(t)\right)^m dt
=h_k(l_k(H_1(p))).
\end{equation*}
This implies the inequality.

Assume the equality holds.
Then by the equality case of \cite[Cor. 3.3.2]{HK}, $ \partial \Omega$ is umbilical and the sectional curvatures of the planes containing the tangent to any geodesic normal to $\partial \Omega$ are equal to $k$ up to the cut point. Also, from the proof above, $c(p)=l_k(H_1(p))$ for all $p\in\partial \Omega$, and so the cut distance is also equal to the focal distance at any point on $\partial \Omega$. i.e. any normal geodesic minimizes the distance to $\partial \Omega$ up to the first focal point.

Suppose $n> 2$. By \cite[Lemma 5.3 (iii)]{HK}\footnote{The assumption that $k>0$ is not used in that result. }, $H_1$ is constant and so $c$ is constant. The equality then becomes $\frac{|\Omega|}{|\partial \Omega|}=h_k(l)$ where $l=c(\Omega)$, i.e. the equality in Corollary \ref{cor: iso1} holds. We conclude that $\Omega$ is isometric to $\mathbb B^n_k(l)$.

Now assume $n=2$ and $k>0$. Without loss of generality we can assume $k=1$. Then \eqref{ros} becomes
\begin{equation}\label{sphere ineq}
\mathrm{Area}(\Omega)\le \int _{C} \left(\sqrt{\kappa^2+1}-\kappa\right)ds
\end{equation}
where $\Omega$ is a domain with boundary $C$ and $\kappa$ is the geodesic curvature of $C$. By \cite[Cor. 3.3.2]{HK}, $\Omega$ is a domain on $\mathbb S^2$, which we regard as the unit sphere in $\mathbb R^3$.
We claim that $C=\partial \Omega$ is connected. Indeed, if $C$ is disconnected, then there exists a minimizing geodesic lying in $\Omega$ which connects the two nearest components of $C$. The midpoint of this geodesic segment $\gamma$ is then a common focal point of the two endpoints, which contradicts the minimality of $\gamma$. By Gauss-Bonnet theorem, \eqref{sphere ineq} is equivalent to $\int_C \sqrt{\kappa^2+1}ds\ge 2\pi$. Notice that if we regard $ C$ to be a space curve in $\mathbb R^3$, then its curvature is $\sqrt{\kappa^2+1}$. Fenchel's theorem (\cite[p. 399]{dc}) states that
\begin{align*}
\int_{ C}\sqrt{\kappa^2+1} \, ds \ge 2\pi
\end{align*}
and the equality holds if and only if $ C$ is a plane convex curve. This is only possible when $ C$ is a geodesic circle on the sphere.

Since the case where $k=0$ is handled by \cite[Theorem 3]{Ros1987}, it remains to consider the case $n=2$ and $k<0$. We may assume $k=-1$. Similar as before $\Omega$ is a topological disk in the hyperbolic plane and \eqref{ros} becomes $\mathrm{Area}(\Omega)\le \int_C (\kappa-\sqrt{\kappa^2-1})ds$, or equivalently by Gauss-Bonnet theorem that
\begin{align*}
\int_C \sqrt{\kappa^2-1}\, ds\le 2\pi.
\end{align*}
Note that $\sqrt{\kappa^2-1}$ is the curvature of $C$ when regarded as a closed spacelike curve in $\mathbb R^{2, 1}$, which contains $\mathbb H^2$ as the hyperboloid $\{(x_0, x_1, x_2):x_0{}^2-x_1{}^2-x_2{}^2=1, x_0>0\}$. We have $\sqrt{\kappa^2-1}>0$ by assumption, so by the corresponding Fenchel theorem \cite[Theorem 1.1]{ye2019closed} in $\mathbb R^{2, 1}$, the equality holds if and only if it is a convex curve on a spacelike plane. Hence there exists $X_0\in \mathbb H^2$ such that $\langle X_0, X\rangle =\mathrm{const. }$ for $X\in C$. Thus $C$ is a geodesic circle centered at $X_0$.

\end{proof}

\begin{corollary} \label{maximizer}
Suppose $(\Omega^n, g)$ is a complete Riemannian manifold with $C^2$ compact boundary. Assume the Ricci curvature of $\Omega$ satisfies $\mathrm{Ric}\ge (n-1)k g$ and $H_1\ge \lambda_k(l)$ on the boundary, then
\begin{align*}
\frac{\mathrm{Vol}(\Omega)}{ \mathrm{Area}(\partial \Omega)}\le \frac {\mathrm{Vol}\left(\mathbb B^n_k (l)\right)}{\mathrm{Area}(\mathbb S^{n-1}_k(l))}.
\end{align*}
The equality holds if and only if $\Omega$ is isometric to $\mathbb B^n_k(l)$.
\end{corollary}
In particular, among all closed embedded minimal hypersurfaces in $\mathbb S^n$, the hemisphere has the largest volume-to-area ratio (cf. \cite[Theorem 2.3]{HK}). When $k\le 0$, by the concavity of $h_k$, the assumption on $H_k$ can also be weakened to $\fint _\Sigma l_k(H_1)\le l$ where $l>\sqrt{-k}$.

As an application, by combining with the ``area-mean curvature characterization of standard spheres'' of Almgren, we give an alternative proof that a $C^2$ isoperimetric domain in $M_k$ is a ball.

\begin{theorem}\label{isop domain}
Let $\Omega$ is compact domain with $C^2$ boundary in $M_k$. Suppose $\Omega$ is a solution to the isoperimetric problem, then it is a geodesic ball.
\end{theorem}

\begin{proof}
It is known that $\partial \Omega$ has constant mean curvature (\cite[Thm. II.1.4]{chavel2001isoperimetric}), and by comparing the principal curvatures at the point of contact with the largest circumscribed sphere, we see that $H_1>\sqrt{-k}$ if $k\le 0$. If $k>0$, we can let $ H_1 =\lambda_k(l)\ge 0$, $l\le \frac{1}{2}\mathrm{diam}(M_k)$, by swapping $\Omega$ with its complement if necessary.

Let $v=|\Omega|$, $r_k$ be the inverse function of $r\mapsto |\mathbb B^n_k(r)|$ and $A_k(r):=|\mathbb S^{n-1}_k(r)|$. Since $ \Omega$ is a solution to the isoperimetric problem, $|\partial \Omega|\le A_k(r_k(v))$.
On the other hand, by \cite[Theorem 1]{almgren1986optimal} when $k\ge 0$, or \cite[Proposition 8]{kleiner1992isoperimetric}\footnote{We can extend this proof to any dimension since we are only working on $M_k$. } when $k<0$, we have $|\partial \Omega|\ge A_k(l)$. So we have $r_k(v)\ge l$, as $A_k$ is increasing on $[0, \frac{1}{2}\mathrm{diam}(M_k)]$.

By Corollary \ref{maximizer}, $\frac{|\Omega|}{|\partial
\Omega|}\le h_k(l)$. Therefore
$\frac{v}{A_k(r_k(v))}\le h_k(l)\le h_k(r_k(v))$. But this must be an equality by \eqref{eq: h}, so by Corollary \ref{maximizer}, $\Omega$ is a geodesic ball.
\end{proof}

Inspired by the Alexandrov theorem, we want to characterize domains with constant cut function.
It is generally not true that a domain with constant cut function is necessarily a ball. A counter-example is the annulus. If we allow $\partial \Omega$ to be only $C^1$ smooth, then a sausage body \cite{chernov2019sausage} is also a domain in $\mathbb R^n$ with constant cut function, which is homeomorphic to but in general different from a ball. It is interesting to know under what additional assumption that $\Omega$ is a geodesic ball.
\begin{proposition}
Suppose $\Omega$ satisfies the assumptions in Theorem \ref{thm: vol c}. Assume $c$ is constant on $\partial \Omega$ and $\lambda_k(c)|\partial \Omega|\le \int_{\partial \Omega} H_1$, then $\Omega$ is isometric to $\mathbb B^n_k(c)$.
\end{proposition}
\begin{proof}
By Theorem \ref{thm: vol c}, $H_1\le \lambda_k(c)$ on $\partial \Omega$ and thus $\int_{\partial \Omega} H_1 \le \lambda_k(c)|\partial \Omega|$.
The assumption then implies $H_1\equiv \lambda_k(c)$. By Corollary \ref{cor: iso1}, $\int_{\partial \Omega}h_k(l_k(H_1))=\int_{\partial \Omega} h_k(c)=h_k(c)|\partial \Omega|\le |\Omega|$. But this is the reverse inequality of Theorem \ref{v h}, so $\Omega$ is isometric to $\mathbb B^n_k(c)$.
\end{proof}
\section{Isoperimetric inequalities involving the extrinsic radius}
\subsection{An inequality involving the extrinsic radius}\label{extrin}
We now turn into another isoperimetric-type inequality.
Firstly, let $M$ be a complete Riemannian manifold (without boundary). For a subset $\Omega$ in $M$, let the extrinsic radius be defined by
\begin{align*}
\mathrm{rad}_M(\Omega)
:=\inf\{r>0: \Omega\subset B(x_0, r) \textrm{ for some } x_0\in M\}
\end{align*}
where $B(x_0, r)$ is the geodesic ball centered at $x_0$ with radius $r$ in $M$. Tautologically $\mathrm{rad}_M(\Omega)<\infty$ means that $\Omega$ is contained in a geodesic ball and $\mathrm{rad}_M(\Omega)=\infty$ if $\Omega$ is not contained in any geodesic ball.

For our purpose, we also define the least average radius of $\Omega$ as follows. Define
\begin{align*}
\mathrm{avrad}_M(\Omega):=\inf\left\{ \fint_{p\in\partial \Omega} d(x_0, p)dS: x_0\in M, \;\Omega\subset B(x_0, r)\textrm{ for some }r \right\}.
\end{align*}
Obviously, we have $\mathrm{avrad}_M(\Omega)\le \mathrm{rad}_M (\Omega)$.

It turns out that if we consider all domains with the same radius and the same curvature upper bound, then the standard ball minimizes the area-to-volume ratio. This looks superficially similar to Corollary \ref{cor: iso1}, but with the inequality sign reversed, and is thus more similar to the Levy-Gromov isoperimetric inequality in spirit. The proof is quite different from and is easier than that of Corollary \ref{cor: iso1}. This type of inequality is called mixed isoperimetric-isodiametric inequality in \cite{mondino2017isoperimetric}.

\begin{theorem}\label{thm: iso2}
Let $M$ be a complete Riemannian manifold (without boundary). Suppose $(\Omega^n, g)$ is a domain in $M$ with compact piecewise $C^1$ boundary and with $\mathrm{rad}_M(\Omega)<\infty$. Assume the sectional curvature of $M$ is bounded above by $k$.
\begin{enumerate}
\item
If $k> 0$ and $\mathrm{rad}_M(\Omega)=L\le \frac{\pi}{\sqrt{k}}$, then
\begin{align}\label{ineq: iso2}
\frac{ \mathrm{Area}(\partial \Omega)}{\mathrm{Vol}(\Omega)}\ge \frac{\mathrm{Area}(\mathbb S^{n-1}_k(L))} {\mathrm{Vol}\left(\mathbb B^n_k (L)\right)}.
\end{align}
\item
If $k\le 0$ and $\mathrm{avrad}_M (\Omega)=L$, then
\begin{align*}
\frac{ \mathrm{Area}(\partial \Omega)}{\mathrm{Vol}(\Omega)}\ge \frac{\mathrm{Area}(\mathbb S^{n-1}_k(L))} {\mathrm{Vol}\left(\mathbb B^n_k (L)\right)}.
\end{align*}
\end{enumerate}
In both cases, the equality holds if and only if $\Omega$ is isometric to $\mathbb B^n_k(L)$.
\end{theorem}

\begin{proof}
We may assume $\Omega$ is contained in a geodesic ball $B$ centered at $x_0$. Let $r=r(x)=d(x_0, x)$ for $x\in M$ and $X=h_k(r)\partial_r$ defined on $B$. Recall that
\begin{equation*} \label{eq: f}
\begin{split}
h_k(r)
:=\frac{|\mathbb B^n_k(r)|}{|\mathbb S^{n-1}_k(r)|}
=\frac{\int_{0}^{r}s_k(t)^{n-1}dt}{s_k(r)^{n-1}}
\end{split}
\end{equation*}
where $s_k$ is defined by \eqref{eq: sk}.
Let $F_k(t)=s_k(t)^{n-1}$.
Then within $B$,
\begin{equation*}
\begin{split}
\mathrm{div} X
=h_k'(r)|\nabla r|^2+h_k(r)\Delta r
=&h_k'(r)+h_k(r)\Delta r\\
\ge&h_k'(r)+ h_k(r)\frac{F_k'(r)}{F_k(r)}\\
=&1-\frac{F_k'(r) \int_{0}^{r}F_k(t)dt}{F_k(r)^2} +h_k(r)\frac{F_k'(r)}{F_k(r)}\\
=&1
\end{split}
\end{equation*}
where the second line follows from Hessian comparison theorem. Note that $X$ is smooth on $\Omega$.
So by divergence theorem, if $\nu$ is the unit outward normal of $\Omega$, then
\begin{equation}\label{ineq: v hk}
\begin{split}
\mathrm{Vol}(\Omega)
\le\int_{\Omega} \mathrm{div} X dV
=\int_{\partial \Omega}\langle X, \nu\rangle dS
\le\int_{\partial \Omega}h_k(r)dS.
\end{split}
\end{equation}
In the case where $k> 0$, we can use the monotone property of $h_k$ (Lemma \ref{lem: mono}) to deduce that $\mathrm{Vol}(\Omega)\le h_k(L)\mathrm{Area}(\partial \Omega)$.
In the case where $k\le 0$, as $h_k$ is concave (Lemma \ref{lem: h convex}), we can apply the Jensen inequality to again deduce that $\mathrm{Vol}(\Omega)\le h_k(L)\mathrm{Area}(\partial \Omega)$ (with different meaning of $L$). In both cases, we have the inequality
\begin{align*}
\frac{ \mathrm{Area}(\partial \Omega)}{\mathrm{Vol}(\Omega)}\ge \frac{\mathrm{Area}(\mathbb S^{n-1}_k(L))} {\mathrm{Vol}\left(\mathbb B^n_k (L)\right)}.
\end{align*}

From the proof above, the equality holds if and only if $r=L$ on $\partial \Omega$ and the Hessian comparison theorem implies that $\Omega=\mathbb B^n_k(L)$.
\end{proof}

\begin{remark}
\begin{enumerate}
\item
Theorem \ref{thm: iso2} is not true if $B(x_0, r)$ in the definition of $\mathrm{rad}_M(\Omega)$ is changed to the metric ball centered at $x_0$ with radius $r$. Consider for example the flat torus $M=\mathbb R^2/\mathbb Z^2$. Let $\Omega_\varepsilon=M\setminus B(x_0, \varepsilon)$. Then $|\partial \Omega_\varepsilon|\to 0$ and $|\Omega_\varepsilon|\to 1$, while $\Omega_\varepsilon$ is always contained in a metric ball of radius which tends to $\frac{1}{\sqrt{2}}$ as $\varepsilon\to 0^+$ (note that $\mathrm{diam}(M)=\frac{1}{\sqrt{2}}$). On the other hand, the RHS of \eqref{ineq: iso2} tends to $2\sqrt{2}$ as $L\to \frac{1}{\sqrt{2}}$.
\item
The assumption in Theorem \ref{thm: iso2} can be weakened to an upper bound for a certain curvature integral, see \cite[Theorem 10]{kwong} for a version of Laplacian comparison in this setting.
\item
We notice that in \cite[Theorem 5.2]{mondino2017isoperimetric}, under some assumptions on the isoperimetric profile and the volumes of balls, it is proved that isoperimetric-isodiametric regions exist, i.e. there exists region $\Omega$ which attains the minimum of $\{\mathrm{Area}(\partial \Omega)h_k(\mathrm{rad}_M(\Omega)): \Omega\subset M, |\Omega|=V\}$ for fixed $V\in (0, |M|)$. They prove regularity results for these regions as well (\cite[Theorem 6.11]{mondino2017isoperimetric}).
\end{enumerate}
\end{remark}
Using Theorem \ref{thm: iso2} and the monotonicity of $h_k$, we obtain a lower bound for the Cheeger's constant.
\begin{definition}[Cheeger \cite{cheeger}]
For a compact Riemannian manifold $M$ with boundary $\partial M\neq \emptyset$, the Cheeger's constant is defined to be
$$h(M):= \inf \left\{ \frac{\mathrm{Area}(\partial \Omega)}{\mathrm{Vol}(\Omega)}\mid \Omega\Subset M \right\}. $$
\end{definition}
\begin{corollary}\label{cheeger}
Let $N$ be a complete Riemannian manifold (without boundary). Suppose $ M$ is a domain in $N$ with smooth compact boundary and with $\mathrm{rad}_N(M)=L<\infty$. Assume the sectional curvature of $N$ is bounded above by $k$ and also $L\le \frac{\pi}{\sqrt{k}}$ if $k>0$.
Then
$h(M)\ge \frac{\mathrm{Area}(\mathbb S^{n-1}_k(L))} {\mathrm{Vol}\left(\mathbb B^n_k (L)\right)} $. In particular, $\lambda_1(M)\ge \frac{1}{4h_k(L)^2}$.
\end{corollary}

As another application, we prove a version of Santal{\'o}-Ya{\~n}ez theorem for family of domains in the hyperbolic space and the Euclidean space. Santal{\'o} and Ya{\~n}ez \cite{average} proved that for a family $\{\Omega_t\}_{t>0} $ of compact $h$-convex (horospherically convex) domains in $\mathbb{H}^{2}$ which expands over the whole space, $\lim _{t \rightarrow \infty} \frac{ \mathrm { Area }(\Omega_t)}{ \mathrm { Length }(\partial \Omega_t)}=1$. This is in contrast with the Euclidean case, where a family of expanding convex domains has $\lim _{t \rightarrow \infty} \frac{ \mathrm { Area }(\Omega_t)}{ \mathrm { Length }(\partial \Omega_t)}=\infty$. Gallego and Revent{\'o}s \cite{asymp} showed that $h$-convexity cannot be relaxed to convexity. Borisenko and Miquel \cite{total} generalized these results to the $n$-dimensional hyperbolic space. We replace the assumption on boundary convexity to that on the cut distance.

\begin{corollary}\label{SY}
Suppose $\{\Omega_t\}_{t>0}$ is a family of compact domains with piecewise $C^2$ boundary in the $n$-dimensional space form $M_k$, where $k\le 0$.
\begin{enumerate}
\item
If $k=0$, assume $\lim_{t\to\infty}\fint_{\partial \Omega_t}c (p)dS =\infty$.
\item
If $k<0$, assume $ \lim_{t\to\infty}c (\Omega_t) =\infty$.
\end{enumerate}
Then
$\lim _{t \rightarrow \infty} \frac{ \mathrm{Area}(\partial \Omega_t)}{ \mathrm{Vol}(\Omega_t)}= \sqrt{-k} (n-1)$.
\end{corollary}
\begin{proof}
We have $\max_{\partial \Omega_t}c(p)\le \mathrm{inradius}(\Omega_t)\le \mathrm{rad}_{M_k}(\Omega_t)\to \infty$ as $t\to\infty$.
Let $l_t=\fint _{\partial \Omega_t}c$ if $k=0$ and $c(\Omega_t)$ if $k<0$, and $L_t=\mathrm{rad}_{M_k}(\Omega_t)$. Then by Theorem \ref{thm: iso2} and Corollary \ref{cor: iso1} (which can be seen to hold for piecewise $C^2$ domains),
$\frac{1}{h_k(L_t)}\le \frac{|\partial \Omega_t|}{|\Omega_t|}\le \frac{1}{h_k(l_t)}$.
So we have
$\lim_{t\to \infty}\frac{|\partial \Omega_t|}{|\Omega_t|}= \sqrt{-k}(n-1) $ by Lemma \ref{lim h}.
\end{proof}

\subsection{Generalizations to higher order mean curvature integrals}
\label{higher}
When $M$ is the Euclidean space or the open hemisphere, an inequality similar to \eqref{ineq: iso2} is true for the ratio of the integrals of the higher order mean curvatures of $\partial \Omega$.

Given a smooth closed hypersurface $ \Sigma \subset M^n $, let $H_j$ be the normalized $j$-th mean curvature of $\Sigma$ (w.r.t. unit outward normal if it makes sense).
Our convention is that for the sphere of radius $r$ in $\mathbb R^n$, $H_j= \frac{1}{r^j}$.
Let $I_j(\Sigma):=\int_{\Sigma}H_j dS$ for a closed hypersurface $\Sigma$ and $I_{-1}(\Sigma)=n\mathrm{Vol}(\Omega)$ where $\Omega$ is the domain bounded by $\Sigma$. (The reason for this notation is due to the Minkowski formula in $\mathbb R^n$.)

\begin{theorem}
For $k>0$, let $M_k^+$ be the open hemisphere of curvature $k$.
Suppose $ \Omega$ is a domain in $M_k^+$ with smooth compact boundary.
Let $i, j\in \mathbb Z$ such that $-1\le i<j\le n-1$. Suppose $H_j>0$ on $\partial \Omega$ and $\mathrm{rad}_{M_k^+}(\Omega)=L$. Then
\begin{align*}
\frac{I_i(\partial \Omega)}{I_j(\partial \Omega)}
\le\frac{I_i(\mathbb S^{n-1}_k(L))}{I_j(\mathbb S^{n-1}_k(L))}.
\end{align*}
The equality holds if and only if $\Omega$ is a geodesic ball of radius $L$.
\end{theorem}

\begin{proof}
We may assume that $k=1$ and that $\Omega$ is contained in a geodesic ball of radius $L$ centered at $p_0$ where $L<\frac{\pi}{2}$. Let $r$ be the distance to $p_0$.
Assume first that $i\ge 0$. We can follow the idea of \cite[Proposition 4.2, Corollary 4.5]{kwong2016extension} to show that for $i\le p<j$ and $q\ge 0$,
\begin{equation*}
\begin{split}
&\int_\Sigma H_p \tan ^q r\\
=& \int_\Sigma H_{p+1} \frac{\tan^{q} r}{\cos r} \langle X, \nu \rangle-\frac{1}{(n-1-p){{n-1}\choose p}}\int_\Sigma \frac{\tan^q r}{\sin^2 r}(q \sec^2 r + \tan^2 r)\langle T_p(X^T), X^T\rangle\\
\le&\int_\Sigma H_{p+1}\tan^{q+1} r,
\end{split}
\end{equation*}
where $X=\sin r\partial_r$, $X^T$ is its tangential part and $T_p$ is the $p$-th Newton's transformation (cf. \cite{kwong2016extension}).
It follows by induction that
$\int_\Sigma H_i \le \int_\Sigma H_j \tan ^{j-i}r$,
from which the result follows.
If $i=-1$, then by \eqref{ineq: v hk} and the above computation,
\begin{align*}
\mathrm{Vol}(\Omega)
\le \int_{\partial \Omega}h_1(r)
\le h_1(L)\int_{\partial \Omega}H_0
\le h_1(L)\tan^j (L)\int_{\partial \Omega}H_j
= \frac{\mathrm{Vol}(\mathbb B^n_1(L))}{I_j(\mathbb S^{n-1}_1(L))}\int_{\partial \Omega}H_j.
\end{align*}
If the equality case holds in any of the above inequalities then $r=L$ on $\partial \Omega$ and so $\Omega$ is a geodesic ball.
\end{proof}

When $M=\mathbb R^n$, Theorem \ref{thm: iso2} can be generalized as follows.
\begin{theorem}
Suppose $ \Omega$ is a domain in $\mathbb R^n$ with smooth compact boundary.
Let $i, j\in \mathbb Z$ such that $-1\le i<j\le n-1$. Suppose $H_j>0$ on $\partial \Omega$ and $\mathrm{rad}_{\mathbb R^n}(\Omega)=L$.
Then
\begin{align*}
\frac{I_i(\partial \Omega)}{I_j(\partial \Omega)}
\le\frac{I_i(\mathbb S^{n-1}(L))}{I_j(\mathbb S^{n-1}(L))}.
\end{align*}
The equality holds if and only if $\Omega$ is a ball of radius $L$.
\end{theorem}
\begin{proof}
Assume $\Omega$ is contained in the ball of radius $l$ centered at $0$ and let $r$ be the distance to $0$.
This then follows from \cite[Theorem 2]{kwong2015monotone} for $i\ge 0$, where it is proved that $\int_{\partial \Omega}H_i\le \int_{\partial \Omega} H_j r^{j-i}$. When $i=-1$, by the proof of Theorem \ref{thm: iso2}, $I_{-1}(\partial \Omega)\le \int _{\partial \Omega}H_0 r$, which is then bounded by $\int_{\partial \Omega
} H_j r^{j+1}$ by \cite[Theorem 2]{kwong2015monotone}, from which the result follows.
\end{proof}
\subsection{Isoperimetric inequality involving the cut locus}\label{cut}
We now discuss what may happen if $\Omega$ is contained in a metric ball instead of a geodesic ball. We use the notation $\mathcal B(x_0, r)$ to denote the closed metric ball centered at $x_0$ with radius $r$, i.e.
\begin{align*}
\mathcal B(x_0, r):=\{x\in M: d(x, x_0)\le r\}.
\end{align*}

\begin{theorem} \label{thm: cut}
Suppose $M$ is a complete $C^\infty$ Riemannian manifold (without boundary) such that its curvature is bounded above by $k$. Suppose $(\Omega^n, g)$ is a domain in $M$ with smooth compact boundary. Suppose $\Omega\subset \mathcal B(x_0, l)$. If $k>0$, we further assume that $l< \frac{\pi}{\sqrt k }$. Then
\begin{align*}
\mathrm{Vol}(\Omega)\le h_k(l)\left(\mathrm{Area}(\partial \Omega\setminus \mathrm{Cut}(x_0))+2 \mathcal H^{n-1}(\mathrm{Cut}(x_0)\cap \overline \Omega)\right)
\end{align*}
where $h_k$ is given by \eqref{eq: h}.
\end{theorem}
\begin{proof}

Let $E_{x_0}=M\setminus \mathrm{Cut}(x_0)$, which is a star-shaped domain in $M$ with respect to $x_0$. Let the cut function $\rho: S_{x_0}M\to (0, \infty]$ be defined as $\rho(v)=\sup\{t>0: \exp_{x_0}(tv)\in E_{x_0} \}$, where $S_{x_0}M=\{v\in T_{x_0}M:|v|=1\}$. Then by \cite[Theorem B]{itoh2001lipschitz}, $\rho$ is locally Lipschitz on which $\rho$ is finite.
From this it is not hard to see that there exists a family of domains $E_\varepsilon\subset E_{x_0}$ with piecewise $C^1$ boundary which are star-shaped with respect to $x_0$, such that $\displaystyle E_\varepsilon\uparrow E_{x_0}$ as $\varepsilon\to 0^+$, i.e. $\displaystyle \bigcup_{\varepsilon>0}E_\varepsilon=E_{x_0}$.
Furthermore, we can assume $\partial E_\varepsilon$ to be the graph of a piecewise $C^1$ function $\rho_\varepsilon: \{v\in S_{x_0}M: \rho(v)<\infty\}\to (0, \infty) $ so that $\mathrm{Area}(\mathrm{Graph} (\rho_\varepsilon|_U))\to \mathcal H^{n-1}(\mathrm{Graph}(\rho|_U))$ as $\varepsilon\to 0^+$ for any relatively compact open set $U$.

Let $\Omega_\varepsilon=\Omega\cap E_\varepsilon$. Let $f$ be defined as in the proof of Theorem \ref{thm: iso2}, then
\begin{align*}
\mathrm{Vol}(\Omega_\varepsilon)
\le\int_{\Omega_\varepsilon}\Delta f dV
=\int_{\partial \Omega_\varepsilon}\langle \nabla f, \nu\rangle dS
=&\int_{\partial \Omega\cap \Omega_\varepsilon}\langle \nabla f, \nu\rangle dS
+\int_{\partial E_\varepsilon\cap \overline \Omega}\langle \nabla f, \nu\rangle dS\\
\le&\int_{\partial \Omega\cap \Omega_\varepsilon}f'(r)dS +\int_{\partial E_\varepsilon\cap \overline \Omega}f'(r)dS\\
=&\int_{\partial \Omega\cap \Omega_\varepsilon}h_k(r)dS +\int_{\partial E_\varepsilon\cap \overline \Omega}h_k(r)dS\\
\le& h_k(l)\mathrm{Area} (\partial \Omega\cap \Omega_\varepsilon) +h_k(l)\mathrm{Area} (\partial E_\varepsilon\cap \overline \Omega).
\end{align*}
Here we have used the fact that $h_k$ is increasing on $[0, l]$ (Lemma \ref{lem: mono}).

We claim that $\displaystyle \lim_{\varepsilon\to 0^+}\mathrm{Area}(\partial E_\varepsilon\cap\overline \Omega)\le 2 \mathcal H^{n-1}(\mathrm{Cut}(x_0)\cap \overline \Omega)$.

By \cite[p. 371, Lemma]{barden1997some} (cf. also \cite{itoh1998dimension}), modulo an $(n-2)$ Hausdorff dimensional set, $\mathrm{Cut}(x_0)$ consists of a disjoint union of smooth hypersurfaces and locally around each point on this set, there are exactly two components of $\partial E_\varepsilon$ which approach $\mathrm{Cut}(x_0)$ from two different sides (because for each point in this set there are exactly two minimal geodesics joining it and $x_0$). Thus $\mathrm{Area}(\partial E_\varepsilon\cap \overline \Omega)\to 2\mathcal {H}^{n-1}(\mathrm{Cut}(x_0)\cap \overline \Omega)$.
So taking $\varepsilon\to 0^+$ in the above inequality, we can get the result.
\end{proof}

By taking $\Omega=M$ in Theorem \ref{thm: cut}, we have the following sharp lower bound for the $(n-1)$-dimensional Hausdorff measure of the cut locus. This complements the results in \cite{itoh1998dimension}, in which it is proved that the cut locus has Hausdorff dimension at most $n-1$.

For our purpose we define the radius at $x_0\in M$ to be
$\mathrm{rad}(x_0):=\sup\{d(x, x_0): x\in M\} $.

\begin{corollary}\label{cor cut}
Suppose $M$ is an $n$-dimensional closed (compact without boundary) manifold and its curvature is bounded from above by $k$. Let $x_0\in M$. If $k>0$, we further assume that $\mathrm{rad}(x_0) <\frac{\pi}{\sqrt{k}}$. Then
\begin{align*}
\frac{\mathrm{Vol}(M)}{2 \mathcal H^{n-1}(\mathrm{Cut}(x_0)) }\le \frac{|\mathbb B^n_k(L)|}{|\mathbb S^{n-1}_k(L)|}
\end{align*}
where $L=\mathrm{rad}(x_0)$.
\end{corollary}

We remark that the inequality is sharp. Take $M=\mathbb RP^n$, which can be modeled as the quotient of the $n$-dimensional Euclidean sphere of radius $r$ by identifying the antipodal points, which has curvature $k=1/r^2$. Clearly we can equip $M$ with the round metric. It is also easy to see that the equality case holds, as for any $x_0\in M$, $\mathcal H^{n-1}(\mathrm{Cut}(x_0))=\frac{1}{2}\left|\mathbb S^{n-1}_k\left(\frac{\pi}{2\sqrt{k}}\right)\right|$, $\mathrm{Vol}(M)=\left|\mathbb B^n_k\left(\frac{\pi}{2\sqrt{k}}\right)\right|$ and $\mathrm{rad}(x_0)=\mathrm{diam}(M)=\frac{\pi}{2\sqrt{k}}$.

\end{document}